\documentclass[a4paper,abstracton]{scrartcl}

\usepackage[utf8]{inputenc}

\usepackage{amsmath}
\usepackage{amsthm}
\usepackage{amssymb}

\usepackage{algorithmic_ok_sok}

\usepackage[usenames,dvipsnames]{color}
\usepackage{graphicx}
\usepackage{subfigure}

\usepackage{amssymb,amsmath,verbatim,tabularx,enumitem,colonequals,graphicx,psfrag,subfigure}

\usepackage{cite}

\newcommand{\cU}{\mathcal{U}}

\newcommand{\R}{\mathbb{R}}
\newcommand{\X}{\mathcal{X}}
\newcommand{\Y}{\mathcal{Y}}

\newcommand{\ep}{\lambda}

\usepackage{authblk}
\usepackage[normalem]{ulem} 

\usepackage{framed}
\usepackage{rotating}

\newtheorem{theorem}{Theorem}
\newtheorem{lemma}[theorem]{Lemma}

\newtheorem{corollary}[theorem]{Corollary}

\newtheorem{example}[theorem]{Example}
\newtheorem*{algorithm}{Algorithm}

\newcommand{\BIGOP}[1]{\mathop{\mathchoice%
{\raise-0.22em\hbox{\huge $#1$}}%
{\raise-0.05em\hbox{\Large $#1$}}{\hbox{\large $#1$}}{#1}}}
\newcommand{\bigtimes}{\BIGOP{\times}}

\allowdisplaybreaks

\begin{document}

\title{Variable-Sized Uncertainty and Inverse Problems in Robust Optimization\thanks{Effort supported by the  Lancaster University Research Committee.}}

\author[1]{Andr{\'e} Chassein\thanks{Email: chassein@mathematik.uni-kl.de}}
\affil[1]{University of Kaiserslautern, Germany}

\author[2]{Marc Goerigk\thanks{Corresponding author. Email: m.goerigk@lancaster.ac.uk}}
\affil[2]{Lancaster University, United Kingdom}

\date{}

\maketitle

\begin{abstract}
In robust optimization, the general aim is to find a solution that performs well over a set of possible parameter outcomes, the so-called uncertainty set. In this paper, we assume that the uncertainty size is not fixed, and instead aim at finding a set of robust solutions that covers all possible uncertainty set outcomes. We refer to these problems as robust optimization with variable-sized uncertainty. We discuss how to construct smallest possible sets of min-max robust solutions and give bounds on their size.

A special case of this perspective is to analyze for which uncertainty sets a nominal solution ceases to be a robust solution, which amounts to an inverse robust optimization problem. We consider this problem with a min-max regret objective and present mixed-integer linear programming formulations that can be applied to construct suitable uncertainty sets.

Results on both variable-sized uncertainty and inverse problems are further supported with experimental data.
\end{abstract}


\section{Introduction}

Robust optimization has become a vibrant field of research with fruitful practical applications, of which several recent surveys give testimonial (see \cite{Aissi2009,RObook,bertsimas-survey,gorissen2015practical,GoeSchoe13-AE,perf}). 
Two of the most widely used approaches to robust optimization are the so-called (absolute) min-max and min-max regret approaches (see, e.g., the classic book on the topic \cite{KouYu97}). For some combinatorial optimization problem of the form
\[ (\mathsf{P}) \qquad \min\ \left\{ c^tx : x\in\X\subseteq\{0,1\}^n \right\} \]
with a set of feasible solutions $\X$, let $\cU$ denote the set of all possible scenario outcomes for the objective function parameter $c$. Then the min-max counterpart of the problem is given as
\[ \min_{x\in\X} \max_{c\in\cU} c^t x\]
and the min-max regret counterpart is given as 
\[ \min_{x\in\X} reg(x,\cU) \]
with
\[ reg(x,\cU) := \max_{c\in\cU} \left( c^tx - opt(c) \right) \]
where $opt(c)$ denotes the optimal objective value for problem $(\mathsf{P})$ with objective $c$.

In the recent literature, the problem of finding suitable sets $\cU$ has come to the focus of attention, see \cite{BertSim04,Bertsimas2004510,Bert09}. This acknowledges that the set $\cU$ might not be ``given'' by a real-world practitioner, but is part of the responsibility of the operations researcher.

In this paper we consider the question how robust solutions change when the size of the uncertainty set changes. We call this approach variable-sized robust optimization and analyze how to find minimal sets of robust solutions that can be applied to any possible uncertainty sets. This way, the decision maker is presented with candidate solutions that are robust for different-sized uncertainty sets, and he can choose which suits him best. Results on this approach for min-max robust optimization are discussed in Section~\ref{sec:inv2}.

As a special case of variable-sized robust optimization, we consider the following question: Given only a nominal problem $(\mathsf{P})$ with objective $\hat{c}$, how large can an uncertainty set become, such that the nominal solution still remains optimal for the resulting robust problem? Due to the similarity in our question to inverse optimization problems, see, e.g., \cite{ahu01,heu04,ali09,Nguyen2015774} we denote this as the inverse perspective to robust optimization. The approach from \cite{Carrizosa2003} is remotely related to our perspective. There, the authors define the robustness of a solution via the largest possible deviation of the problem coefficients in a facility location setting. Our approach can also be used as a means of sensitivity analysis. Given several solutions that are optimal for some nominal problem, the decision maker can choose one that is most robust in our sense. In the same vein, one can check which parts of a solution are particularly fragile, and strengthen them further. This approach is presented for min-max regret in Section~\ref{sec:inv1}.

Our paper closes with a conclusion and discussion of further research directions in Section~\ref{sec:con}.

\section{Variable-Sized Min-Max Robust Optimization}
\label{sec:inv2}

In this section we analyze how optimal robust solutions change when the size of the uncertainty set increases. We assume to have information about the midpoint (nominal) scenario $\hat{c}$, and the shape of the uncertainty set $\cU$. The actual size of the uncertainty set is assumed to be uncertain.

More formally, we assume that the uncertainty set is given in the form
\begin{align*}
\mathcal{U}_\ep = \{\hat{c}\} + \lambda B
\end{align*}
where $B$ is a convex set containing the origin and $\hat{c}$ the midpoint of the uncertainty set. The parameter $\lambda\geq 0$ is an indicator for the size of the uncertainty set. For $\lambda = 0$, we have $\cU_0 = \{\hat{c}\}$, i.e.,  the nominal problem, and for $\lambda \rightarrow \infty$ we obtain the extreme case of complete uncertainty. 

We consider the min-max robust optimization problem
\begin{align*}
\min_{x \in \X} \max_{c \in \cU_\ep} c^tx. \tag{$\mathsf{P}(\ep)$}
\end{align*}
The goal of variable-sized robust optimization is to compute a minimal set $\mathcal{S} \subset \X$ such that for any $\lambda\geq 0$, $\mathcal{S}$ contains a solution that is optimal for $\mathsf{P}(\ep)$. Note that for $\ep = 0$, the set $\mathcal{S}$ must contain a solution $\hat{x}$ that is optimal for the nominal problem. By increasing $\ep$, we trace how this solution needs to change with increasing degree of uncertainty.

In the following we denote by $T$ the time that is necessary to solve the nominal problem $\mathsf{P}$, and assume that $\hat{c} \ge 0$. We derive general complexity results and apply them to the shortest path problem for a more detailed analysis. We denote by $\mathcal{P}$ the set of all paths from a start node $s$ to an end note $t$, and for a path $P\in\mathcal{P}$, we write $c(P) = \sum_{e\in P} c_e$ for its costs.

We relate our general approach to a bicriteria optimization problem in Section~\ref{con-sub0}.
Section~\ref{con-sub1} considers the case where $B$ is a symmetric hyperbox, i.e., $B=\bigtimes_{i\in[n]} [-a_i,a_i]$. In Section~\ref{con-sub2} we then consider the more general case where $B=\{c: \Vert c\Vert \leq 1\}$ is the unit ball of some norm.

\subsection{Relation to Bicriteria Optimization}
\label{con-sub0}

In this section we investigate the close relation between bicriteria optimization and the variable-sized robust optimization problem. We reformulate the objective function of $\mathsf{P}(\ep)$:
\begin{align*}
\max_{c \in \cU_\ep} c^tx = \hat{c}^tx + \max_{c \in \lambda B} c^tx =  \hat{c}^tx + \max_{\tilde{c} \in B} \lambda \tilde{c}^tx = \hat{c}^tx + \lambda \max_{\tilde{c} \in B} \tilde{c}^tx= f_1(x) + \lambda f_2(x)
\end{align*}
where $f_1(x) = \hat{c}^tx$ and $f_2(x) = \max_{c \in B}c^tx$. It is immediate that the variable-sized robust optimization problem is closely related to the bicriteria optimization problem:
\[\min_{x \in \mathcal{X}} \begin{pmatrix} f_1(x) \\ f_2(x) \end{pmatrix}\]

We define the map $F: \mathcal{X} \rightarrow \mathbb{R}^2_+, F(x)=(f_1(x),f_2(x))^t$ which maps every feasible solution in the objective space. Further, we define the polytope $\mathcal{V}= \operatorname{conv}(\{F(x): x \in \mathcal{X}\}) + \mathbb{R}^2_+$. We call a solution $x \in \mathcal{X}$ an \emph{efficient extreme solution} if $F(x)$ is a vertex of $\mathcal{V}$. Denote the set of all efficient extreme solutions with $\mathcal{E}$. We call two different solutions $x\neq x'$ \emph{equivalent} if $F(x)=F(x')$. Let $\mathcal{E}_{\min} \subset \mathcal{E}$ be a maximal subset such that no two solutions of $\mathcal{E}_{\min}$ are equivalent. The next lemma gives the direct relation between $\mathcal{E}_{\min}$ and the variable-sized robust optimization problem. 

\begin{lemma}
$\mathcal{E}_{\min}$ is a solution of the variable sized robust optimization problem.
\label{lem_e_min}
\end{lemma}
\begin{proof}
We need to prove two properties:
\begin{enumerate}
\item [(I)] For every $\lambda\geq 0$ there exists a solution in $\mathcal{E}_{\min}$ which is optimal for $\mathsf{P}(\ep)$.
\item [(II)]$\mathcal{E}_{\min}$ is a smallest possible set with property (I).
\end{enumerate}
(I) Let $\lambda\geq 0$ be fixed. We transfer the problem $\mathsf{P}(\ep)$ in the objective space. The optimal value of $\mathsf{P}(\ep)$ is equal to the optimal value of problem $\mathsf{O}(\lambda)$ since each optimal solution of this problem is a vertex of $\mathcal{V}$.
\begin{align*}
\min_{v \in \mathcal{V}} v_1 + \lambda v_2 \tag{$\mathsf{O}(\lambda)$}
\end{align*}
Let $v^*$ be the optimal solution of $\mathsf{O}(\lambda)$. By definition, $\mathcal{E}_{\min}$ contains a solution $x^*$ with $F(x)=v^*$, i.e., an optimal solution for $\mathsf{P}(\ep)$. 

\noindent
(II) Note that there is a one-to-one correspondence between vertices of $\mathcal{V}$ and solutions in $\mathcal{E}_{\min}$. Since for each vertex $v'$ of $\mathcal{V}$ a $\lambda'\geq0$ exists such that $v'$ is optimal for $\mathsf{O}(\lambda)$, it follows that $\mathcal{E}_{\min}$ is indeed minimal. Note that it is important to ensure that $\mathcal{E}_{\min}$ contains no equivalent solutions.
\end{proof}

\noindent
We use the following general procedure to compute $\mathcal{E}_{\min}$. First, we find the efficient extreme solution $x_1^*$ which minimizes $f_1$. This can be done by solving a weighted sum problem between first and second objective function where the weight for the second objective function is chosen sufficiently small. Second, we find the efficient extreme solutions $x_2^*$ which minimizes $f_2$. Next we apply the function EXPLORE($x_1^*,x_2^*$) which is recursively defined.
 
\begin{algorithm}
\begin{algorithmic}[1]{EXPLORE($x_1,x_2$)}
\INPUT{Two efficient extreme solutions $x_1$ and $x_2$ with $f_1(x_1) < f_1(x_2)$}
\OUTPUT{All efficient extreme solutions $y$ with $f_1(x_1) < f_1(y) < f_1(x_2)$}
\STATE Choose $\lambda$ such that $f_1(x_1)+\lambda f_2(x_1) =  f_1(x_2)+\lambda f_2(x_2)$
\STATE $x' := \operatorname{argmin}_{x \in \mathcal{X}} f_1(x)+\lambda f_2(x)$
\IF{$f_1(x')+\lambda f_2(x') < f_1(x_1)+\lambda f_2(x_1)$}
	\STATE \RETURN $\{x'\}$ $\cup$ EXPLORE($x_1,x'$) $\cup$ EXPLORE($x',x_2$)
\ELSE
	\STATE \RETURN $\emptyset$
\ENDIF
\end{algorithmic}
\end{algorithm}

\noindent
\textbf{Remark}: It is possible that the described procedure generates some additional solutions $x'$ with $F(x')$ on an edge of $\mathcal{V}$ and not on the vertex of $\mathcal{V}$. These solutions are not efficient extreme solution. It is easy to show that at most one additional solution for each edge of $\mathcal{V}$ is generated. Hence, the number of additional generated solutions is bounded by the number of efficient extreme solutions. 

\begin{theorem}
The variable-sized robust problem can be solved in $O(|\mathcal{E}_{\min}|T)$.
\label{thm:sol_time}
\end{theorem}
\begin{proof}
We need to show that the described procedure solves at most $O(|\mathcal{E}_{min}|)$ many nominal problems. For simplicity, we assume that no solutions $x'$ exist with $F(x')$ on an edge of $\mathcal{V}$. Each solution of a nominal problem in the procedure, discovers either a new vertex or a new edge of $\mathcal{V}$. Since $\mathcal{V} \subset \mathbb{R}^2$, we have that the number of edges is in $O(|\mathcal{E}_{\min}|)$. Therefore, the number of nominal problems which needs to be solved is in $O(|\mathcal{E}_{\min}|)$.
\end{proof}

\subsection{Hyperbox-Shaped Uncertainty}
\label{con-sub1}

In this section we consider the case that $B$ is a hyperbox, i.e., $B=\bigtimes_i [-a_i,a_i]$. We distinguish three different types of hyperboxes. 

\begin{enumerate}
\item $B= \bigtimes_i [-\hat c_i,\hat c_i]$ (proportional growth)
\item $B= \bigtimes_i [-1,1]$ (constant growth)
\item $B= \bigtimes_i [-d_i,d_i]$ (arbitrary growth)
\end{enumerate}

In the case of proportional growth the size of the box purely depends on the given midpoint $\hat{c}$, for constant growth the size of the box is uniform for all elements of $\mathcal{X}$, and for arbitrary growth (the most general case) the size of the box is independent of $\hat{c}$.

Note that the different growth rates lead to different objective functions for problems~$\mathsf{P}(\lambda)$:
\begin{align*}
\max_{c \in \cU_\ep} c^tx = \begin{cases}(1+\ep)\hat{c}^tx, &\text{proportional growth} \\ \hat{c}^tx + \lambda \Vert x\Vert_1, &\text{constant growth} \\ \hat{c}^tx+ \lambda d^tx, &\text{arbitrary growth}   \end{cases}
\end{align*}

The case of proportional growth is straightforward. Since the midpoint solution is optimal for any $\lambda\geq 0$, the following theorem follows immediately.

\begin{theorem}
The variable-sized robust problem with proportional growth can be solved in time $T$ and $|\mathcal{S}|=1$.
\end{theorem}

\noindent
We analyze constant and arbitrary growth in the following.

\subsubsection{Constant Growth}
The case of constant growth is more involved. Consider the bicriteria optimization problem 
\[\min_{x \in \mathcal{X}} \begin{pmatrix} \hat{c}^tx \\ \Vert x\Vert_1 \end{pmatrix}.\]

\begin{lemma}
We have that $|\mathcal{E}_{\min}| \leq n$ .
\label{lem:size_e_min_constant_growth}
\end{lemma}
\begin{proof}
If $0 \in \mathcal{X}$, $\mathcal{E}_{\min} = \{0\}$, since $\hat{c} \geq 0$. Otherwise, $\Vert x\Vert_1 \in \{1,\dots,n\}$ for each $x \in \mathcal{X}$. Observe that for all solution pairs $x,y \in \mathcal{E}_{\min}$ it holds that $\Vert x\Vert_1 \neq \Vert y\Vert_1$. This yields the claimed bound $|\mathcal{E}_{\min}| \leq n$. 
\end{proof}

\noindent
Combining Lemma~\ref{lem:size_e_min_constant_growth} and Theorem~\ref{thm:sol_time} we get the following theorem.

\begin{theorem}
The variable-sized robust problem with constant growth can be solved in $O(nT)$ and $|\mathcal{S}|\leq n$.
\label{thm:constant_growth}
\end{theorem}

\paragraph{Application to the shortest path problem.}
We consider in more detail the implications of Theorem~\ref{thm:constant_growth} to the shortest path problem on a graph $G=(V,E)$ with $N$ nodes and $M$ edges. The corresponding bicriteria optimization problem is 
\[\min_{P \in \mathcal{P}} \begin{pmatrix} \hat{c}(P) \\ |P| \end{pmatrix}\]

Note that it suffices to consider simple paths for this bicriteria optimization problem as we assume that all edge costs are positive. As each simple path contains at most $N$ edges $|P| \in \{1,\dots,N\}$, the cardinality of $\mathcal{E}_{\min}$ is bounded by $N$. The computation of this set can be done at once using a labeling algorithm that stores for each path $Q$ from $s$ to $v$ the cost of the path $\hat{c}(Q)$ and the number of edges contained in the path $|Q|$. Note that at each node at most $N$ labels needs to be stored, this ensures the polynomial running time of the labeling algorithm. An alternative approach is to use the described procedure to compute $\mathcal{E}_{\min}$. Each weighted sum computation corresponds to a shortest path problem.

\begin{theorem}
The variable-sized robust shortest path problem for a graph $G=(V,E)$ with $|V|=N$ and $|E|=M$ with constant growth can be solved in $O(NM+N^2 \operatorname{log}(N))$ and $|\mathcal{S}|\leq N$.
\end{theorem}
\begin{proof}
The nominal problem can be solved by Dijkstra's algorithm in $O(M+N \operatorname{log}(N))$.
\end{proof}

\subsubsection{Variable Growth}

For variable growth we consider the general bicriteria optimization problem
\[\min_{x \in \mathcal{X}} \begin{pmatrix} \hat{c}^tx \\ d^tx \end{pmatrix}.\]

Define by $K$ a bound for the size of the set $\mathcal{E}_{\min}$. Unfortunately, for this general problem there exists no such polynomial bound $K$. 

\begin{theorem}
The variable-sized robust problem with variable growth can be solved in $O(KT)$ time and $|\mathcal{S}|\leq K$.
\label{thm_var_growth}
\end{theorem}

\paragraph{Application to the shortest path problem.}
Again we investigate the consequences of Theorem~\ref{thm_var_growth} to the variable-sized robust shortest path problem on a graph $G=(V,E)$ with $N$ nodes and $M$ edges. Carstensen presents in her PhD thesis~\cite{carstensen1983complexity} bicriteria shortest path problems in which $|\mathcal{E}_{\min}| \in 2^{\Omega(\operatorname{log}^2(N))}$, which is not polynomial in $N$. Further, she proved for acylic graphs that $K \in O(n^{\operatorname{log}(n)})$ which is subexponential. Note that this result can directly been applied to the variable-sized robust shortest path problem with variable growth.

\begin{theorem}
The variable-sized robust shortest path problem on acyclic graphs with variable growth can be solved in subexponential time and $|\mathcal{S}|$ is also subexponential.
\end{theorem}

\noindent
If we further restrict the graph class of $G$, the following results can be obtained. 

\begin{theorem}
Let $G$ be a series-parallel graph with $N$ nodes and $M$ arcs. Then, $K \leq M-N+2$, and the variable-sized robust shortest path problem can be solved in polynomial time.
\label{thm:E_bound_spgraph}
\end{theorem}
\begin{proof}
We do a proof by an induction over the depth~$D(G)$ of the decomposition tree of graph~$G$.  
For the induction start, note that if $D(G)=1$, $G$ only consists of a single edge. Hence, there exists exactly one $s-t$ path, which is obviously also an extreme efficient solution. Therefore, $K=1$. There are two nodes and one arc, i.e., $N=2$ and $M=1$. Hence, $K\le M-N+2$ holds. \\

\noindent
We distinguish two cases for the induction step: \\

\noindent
\textit{Case 1:} $G$ is a parallel composition of two series-parallel graphs $G_1$ and $G_2$.

Every path that is an extreme efficient solution for the shortest path problem in graph $G$ is then either completely contained in $G_1$ or $G_2$, and, hence, must also be an extreme efficient solution of the shortest path instance described by $G_1$ or $G_2$. Therefore, the number of extreme efficient paths in $G_1$ plus the number of extreme efficient paths in $G_2$ is an upper bound for the number of extreme efficient paths in $G$. Note that $D(G_1)<D(G)$ and $D(G_2)<D(G)$. Hence, we can apply the induction hypothesis to $G_1$ and $G_2$. Denote by $K_i$ the number of extreme efficient paths in $G_i$, by $N_i$ the number of nodes, and by $M_i$ the number of edges of $G_i$ for $i=1,2$. We have that
\begin{align*}
K \leq K_1+K_2 &\leq (M_1-N_1+2) + (M_2-N_2+2)\\
&= M_1+M_2-(N_1+N_2-2)+2 = M-N+2
\end{align*}

\noindent
\textit{Case 2:} $G$ is a series composition of two series-parallel graphs $G_1$ and $G_2$.

Note that for every extreme efficient path, there exists an open interval $(\underline{\lambda},\overline{\lambda})$ such that this path is the unique optimal path with respect to the weight function $\lambda c + (1-\lambda)d$ for all $\lambda \in (\underline{\lambda},\overline{\lambda})$. Hence, we can find an ordering of all extreme efficient paths with respect to the $\lambda$ values. Note that all extreme efficient paths in $G$ must consist of extreme efficient paths of $G_1$ and $G_2$. The extreme paths of $G$ can be obtained in the following way. We start with the shortest paths with respect to $c$ in $G_1$ and $G_2$ and combine these two paths. Next we decrease the value of $\lambda$ until either the extreme path in $G_1$ or $G_2$ changes. This change will define a new extreme path. We continue until all extreme paths of $G_1$ and $G_2$ are contained in at least one extreme path. Note that $D(G_1)<D(G)$ and $D(G_2)<D(G)$. Hence, we can apply the induction hypothesis to $G_1$ and $G_2$. Note that the 
number of changes is bounded by the number of extreme efficient paths minus~$1$. Since, for every such change, we get one additional extreme efficient path, we get in total 
\begin{align*}
K \leq \ & 1+(M_1-N_1+2-1)+(M_2-N_2+2-1)\\
&= M_1+M_2-(N_1+N_2-1)+2\\
&=M-N+2.
\end{align*}
\end{proof}

Using a similar proof technique we can bound $K$ also for layered graphs. A layered graph consists of a source node $s$ and a destination node $t$ and $\ell$ layers, each layer consists of $w$ nodes. Node $s$ is fully connected to the first layer, the nodes of the $i^{\text{th}}$ layer are fully connected to the nodes of the $(i+1)^{\text{th}}$ layer, and the last layer is fully connected to node $t$.

\begin{theorem}
Let $G$ be a layered graph with $\ell$ layers and width $w$. Then, $K \leq (2w)^{ \lceil \log(\ell+1) \rceil}.$
\label{thm:E_bound_layered}
\end{theorem}
\begin{proof}
Without loss of generality we can assume that $\ell = 2^k-1$ for some $k\in \mathbb{N}$. Then, the claimed bound simplifies to $K \leq (2w)^{k}$. Denote by $E(t)$ the number of extreme efficient paths in $G$ with $2^t-1$ many layers. The base case is given by $E(0)=1$, since the corresponding graph consists of a single edge. To bound $E(k)$ we make the following observation. Each $s-t$ path contains a single node $j \in [w]$ from the middle layer (the middle layer is the $(2^{k-1}-1)^\text{th}$ layer). Hence, we can separate all $s-t$ paths into $w$ disjoint different classes. Let $E_j$ be the number of extreme efficient solutions in each class $j \in [w]$. The same arguments as in Case~2 in the proof of Theorem~\ref{thm:E_bound_spgraph} can be used to show that $E_j$ is at most the number of extreme efficient paths from $s$ to $j$ plus the number of extreme efficient paths from $j$ to $t$ minus $1$. Observe further that all paths from $s$ to a node $j$ on the middle layer and all paths from $j$ to $t$ are contained in a layered graph with width $w$ and $2^{k-1}-1$ layers. Using these arguments we obtain the bound: $E(k) \leq w(2E(k-1)-1)$. Using this bound we can derive the claimed bound
\begin{align*}
K= E(k) \leq w(2E(k-1)-1) \leq 2wE(k-1) \leq  (2w)^k E(k-k) = (2w)^k
\end{align*}

\end{proof}

\noindent
\textbf{Remark:} The bound proved in Theorem~\ref{thm:E_bound_layered} is subexponential for general layered graphs. But if $\ell$ or $w$ is assumed to be fixed, the bound is polynomial: Since $N \approx \ell w$, $w \approx \frac{N}{\ell}$ and $\ell \approx \frac{N}{w}$, we have for $\ell$ fixed $K \in O\left( (\frac{N}{\ell})^{\log(\ell)}\right) \in O\left(N^{\log(\ell)}\right)$ and, conversely, for $w$ fixed $K \in O\left((2w)^{\log(\frac{N}{w})}\right) \in O\left(N^{\operatorname{log}(w)+1}\right)$.

\subsection{Norm Balls}
\label{con-sub2}

In this section we consider the case that $B$ is a unit ball of some norm $\Vert\cdot\Vert$, i.e., $B=\{c \mid \Vert c\Vert\leq 1 \}$. We investigate in more detail four different norms. 

\begin{enumerate}
\item The generalized infinity norm: $\Vert c\Vert = \max_{i}\frac{|c_i|}{d_i}$, where $d>0$. 
\item The generalized Manhattan norm: $\Vert c\Vert = \sum_i \frac{|c_i|}{d_i}$, where $d>0$. 
\item The generalized Euclidean norm: $\Vert c\Vert = \sqrt{\sum_i \frac{c_i^2}{d_i}}$, where $d>0$.  
\item An ellipsoidal norm: $\Vert c\Vert = \sqrt{c^tQ^{-1}c}$, where $Q\succ 0$ is a positive definite matrix.
\end{enumerate}

Note that we already discussed the generalized infinity norm in Section~\ref{con-sub1}, since the normal ball of the generalized infinity norm is a hyperbox. Each norm leads to different objective values for $\mathsf{P}(\lambda)$:
\begin{align*}
\max_{c \in \cU_\ep} c^tx = \begin{cases} \hat{c}^tx + \lambda d^tx, &\text{generalized infinity norm} \\ \hat{c}^tx + \lambda \max_{i}d_ix_i , &\text{generalized Manhattan norm} \\ \hat{c}^tx + \lambda \sqrt{\sum_i d_i x_i^2}, &\text{generalized Euclidean norm}  \\ \hat{c}^tx + \lambda \sqrt{x^tQx} , &\text{ellipsoidal norm}  \end{cases}
\end{align*}

\subsubsection{Generalized Manhattan Norm}

We use the following bicriteria optimization problem to investigate the problem

\[\min_{x \in \mathcal{X}} \begin{pmatrix} \hat{c}^tx \\ \max_{i}d_ix_i \end{pmatrix}.\]

We are again interested in the set $\mathcal{E}_{\min}$ of efficient extreme solutions of this problem. Equivalent to the case of constant growth we obtain the following lemma and theorem. 

\begin{lemma}
We have that $|\mathcal{E}_{\min}| \leq n$.
\label{size_e_min_2}
\end{lemma}
\begin{proof}
If $0 \in \mathcal{X}$, $\mathcal{E}_{\min} = \{0\}$, since $\hat{c} \geq 0$. Otherwise, $ \max_{i}d_ix_i \in \{d_1,\dots,d_n\}$ for each $x \in \mathcal{X}$. Observe that for all solution pairs $x,y \in \mathcal{E}_{\min}$ it holds that $\max_{i}d_ix_i \neq \max_{i}d_iy_i$. This yields to the claimed bound $|\mathcal{E}_{\min}| \leq n$. 
\end{proof}

\begin{theorem}
The variable-sized robust problem with the generalized Manhattan norm can be solved in $O(nT)$ and $|\mathcal{S}|\leq n$.
\label{thm:gen_man_norm}
\end{theorem}

\paragraph{Application to the shortest path problem.}
We consider in more detail the implications of Theorem~\ref{thm:gen_man_norm} to the shortest path problem on a graph $G=(V,E)$ with $N$ nodes and $M$ edges. The corresponding bicriteria optimization problem is 
\[\min_{P \in \mathcal{P}} \begin{pmatrix} \hat{c}(P) \\ \max_{e \in P} d_e \end{pmatrix}\]

The set $\mathcal{E}_{\min}$ is bounded by the number of different values of $d_e$, $e \in E$,  which is bounded by $M$. The computation of this set can be done at once using a labeling algorithm that stores for each path $Q$ from $s$ to $v$ the cost of the path $\hat{c}(Q)$ and the most expensive edge of $Q$ with respect to cost function $d$. Note that at each node at most $M$ labels needs to be stored, this ensures the polynomial running time of the labeling algorithm. An alternative approach is the described procedure to compute $\mathcal{E}_{\min}$. If we use Dijkstra's algorithm to solve the nominal problem we obtain the following running time.

\begin{theorem}
The variable-sized robust shortest path problem for a graph $G=(V,E)$ with $|V|=N$ and $|E|=M$ with the generalized Manhattan norm can be solved in $O(M^2+NM \operatorname{log}(N))$ and $|\mathcal{S}|\leq M$.
\end{theorem}

\noindent
For more details on bicriteria problems of this type, we refer to \cite{gorski2012generalized}.

\subsubsection{Generalized Euclidean Norm}

Note that $x_i \in \{0,1\}$ since we consider combinatorial optimization problems. Hence we can simplify the corresponding objective function of $\mathsf{P}(\ep)$ to 
\begin{align*}
\max_{c \in \cU_\ep} c^tx = \hat{c}^tx + \lambda \sqrt{\sum_i d_i x_i^2} = \hat{c}^tx + \lambda \sqrt{d^tx}.
\end{align*}

\noindent
The corresponding bicriteria optimization problem has the form
\[\min_{x \in \mathcal{X}} \begin{pmatrix} \hat{c}^tx \\ \sqrt{d^tx}  \end{pmatrix}.\]

\noindent
The next lemma allows us to make use of results obtained for variable growth.

\begin{lemma}
Each efficient extreme solution of 
\[ \min_{x \in \mathcal{X}} (\hat{c}^tx,\sqrt{d^tx}) \tag{$\mathsf{P}_1$}\]
is an efficient extreme solution of
\[ \min_{x \in \mathcal{X}} (\hat{c}^tx,d^tx) \tag{$\mathsf{P}_2$}.\]
\label{lem_ext_eff_to_ext_eff}
\end{lemma}
\begin{proof}
We define the maps $cd: \mathcal{X} \rightarrow \mathbb{R}^2, cd(x)=(\hat{c}^tx,d^tx)$ and $cd': \mathcal{X} \rightarrow \mathbb{R}^2, cd'(x)=(\hat{c}^tx,\sqrt{d^tx})$. Further denote by $C_{\min} = \min_{x \in \mathcal{X}} \hat{c}^tx$ and $D_{\min} = \min_{x \in \mathcal{X}} d^tx$. Let $x'$ be an efficient extreme solution of $\mathsf{P}_1$. If $cd(x')_1=C_{\min}$ or $cd(x')_2=D_{\min}$, it is straightforward to show that $x'$ is also efficient extreme for $\mathsf{P}_2$.

Hence, we assume in the following that $cd(x')_1>C_{\min}$ and $cd(x')_2>D_{\min}$. We assume, for the sake of contradiction, that $x'$ is not an efficient extreme solution of $\mathsf{P}_2$. This means that two other solutions $x_1,x_2 \in \mathcal{X}$ and an $\alpha \in [0,1]$ exists such that $\alpha cd(x_1)_1 + (1-\alpha) cd(x_2)_1 = cd(x')_1$ and $\alpha cd(x_1)_2 + (1-\alpha) cd(x_2)_2 \leq cd(x')_2$ (see Figure~\ref{fig:proof_help}). It follows that $\alpha cd'(x_1)_1 + (1-\alpha) cd'(x_2)_1 = cd'(x')_1$ and
\[ \alpha cd'(x_1)_2 + (1-\alpha) cd'(x_2)_2 \leq \sqrt{\alpha cd(x_1)_2 + (1-\alpha) cd(x_2)_2} \leq \sqrt{cd(x')_2} =  cd'(x')_2,\]
since the square root is a concave and monotone function. This yields the desired contradiction, since $x'$ is an efficient extreme solution for $\mathsf{P}_1$.
\end{proof}

\begin{figure}
\centering
\begin{center}
\includegraphics[trim = 50mm 25mm 55mm 225mm, clip,width=\textwidth]{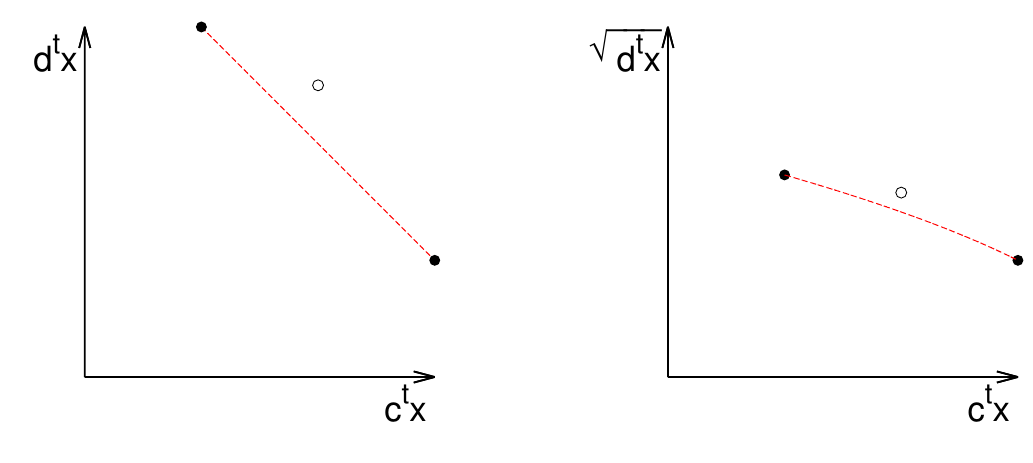}
\end{center}
\caption{In the left (right) coordinate system $cd$ ($cd'$) is plotted for three different solutions. The solution represented by the empty circle is not an extreme efficient  solution of problem $\mathsf{P}_2$ (left plot). We can conclude (right plot) that it can not be an extreme efficient solution of problem $\mathsf{P}_1$.}
\label{fig:proof_help}
\end{figure}

To solve the variable-sized robust shortest path problem with generalized Euclidean norm it suffices to compute the solution $\mathcal{S}$ of the corresponding variable growth problem and to remove all solutions which are not efficient extreme for the original bicriteria optimization problem.

\subsubsection{Ellipsoidal Norm}

Note that in this setting we can not simplify the quadratic expression in the objective function. Under slight modifications the proof of Lemma~\ref{lem_ext_eff_to_ext_eff} can be used to show a similar lemma:

\begin{lemma}
Each efficient extreme solution of $\min_{x \in \mathcal{X}} (\hat{c}^tx,\sqrt{x^tQx})$ is an efficient extreme solution of $\min_{x \in \mathcal{X}} (\hat{c}^tx,x^tQx)$.
\end{lemma}

\noindent
Hence, the following bicriteria optimization problem needs to be considered
\[\min_{x \in \mathcal{X}} \begin{pmatrix} \hat{c}^tx \\ x^tQx  \end{pmatrix}.\]

\noindent
The corresponding weighted sum problem of this bicriteria problem can be written as
\[\min_{x \in \mathcal{X}} \hat{c}^tx + \lambda x^tQx =  x^t\hat{C}x + \lambda x^tQx = x^t(\hat{C}+\lambda Q) x\]
where $\hat{C}=\operatorname{diag}(\hat{c}_1,\dots,\hat{c}_n)$ is a diagonal matrix. Note that $\hat{C}+\lambda Q$ is positive definite. We see that, it suffices to solve a list of nominal problems with quadratic objective function to solve the variable-sized robust problem with ellipsoidal norm. But in some cases the transformation to a quadratic problem is not necessary and the problem can be solved directly without removing the square root. We present this in the following for the shortest path problem.

\paragraph{Application to the shortest path problem.}

The feasible set of the shortest path problem, i.e. all incidence vectors of $s$-$t$ paths can be represented in the following way
\[ \mathcal{X} = \left\{ x \in \{0,1\}^M: \sum_{e \in \delta^+(v)} x_e - \sum_{e \in \delta^-(v)} x_e = b_v \ \forall v \in V \right\} \text{ with } b_v = \begin{cases} +1, &\text{ if } v=s \\ -1, &\text{ if } v=t \\  0, &\text{ else } \end{cases} \]

where $\delta^+(v)$($\delta^-(v)$) is the set of are all edges leaving (entering) $v$. Note that the constraints defining $\mathcal{X}$ do not forbid cycles for the $s$-$t$ paths. Nevertheless, we can use this set for the optimization problem, since all solutions containing cycles are suboptimal. The weighted sum problem which needs to be solved to compute the set of extreme efficient solutions can be represented by the following mixed integer second order cone programming problem.
\begin{align*}
\min\ & c^tx + r \\
\text{s.t. } & x^tQx \leq r^2  \\
& x \in \mathcal{X}
\end{align*}
The computational complexity of this problem is $\textsf{NP}$-complete and even $\textsf{APX}$-hard as shown in \cite{QSPP}. Nevertheless, real world instances can be solved exactly using modern solvers. We conclude the section with a case study of the variable-sized robust shortest path problem with ellipsoidal uncertainty.

\subsection{Case Study}

In this case study we consider the problem of finding a path through Berlin, Germany. We use a road network of Berlin which consists of $12,100$ nodes and $19,570$ edges. The data set was also used in \cite{jahn2005system}, and taken from the collection \cite{dataset}.
We use the following probabilistic model to describe the uncertain travel times of each road segment.
\begin{align*}
c = \hat{c} + L\xi 
\end{align*}
where $\xi \sim \mathcal{N}(0,I)$ is a $k$-dimensional random vector which is multivariate normal distributed, $L \in \mathbb{R}^{M \times k}$, and $\hat{c}$ is taken from the data set. Under this assumption $c$ is also multivariate normally distributed with mean $\hat{c}$ and variance $LL^t$. Note that the most likely realization of $c$ form an ellipsoid. Protecting against these realizations we obtain an ellipsoidal uncertainty set.

Entries of $L$ are chosen such that the road segments in the center of Berlin tend to be affected by more uncertainty. We compute all solutions to the variable-sized robust problem, solving the resulting mixed integer second order cone programming problem with Cplex v.12.6. Computation times were less than three minutes using a quad core processor with 3.2 GHz.
The solution of the variable-sized robust problem contains $11$ different paths. These paths are shown in Figure~\ref{fig_berlin}.

\begin{figure}
\begin{center}
 \includegraphics[trim = 50mm 23mm 46mm 190mm, clip,width=\textwidth]{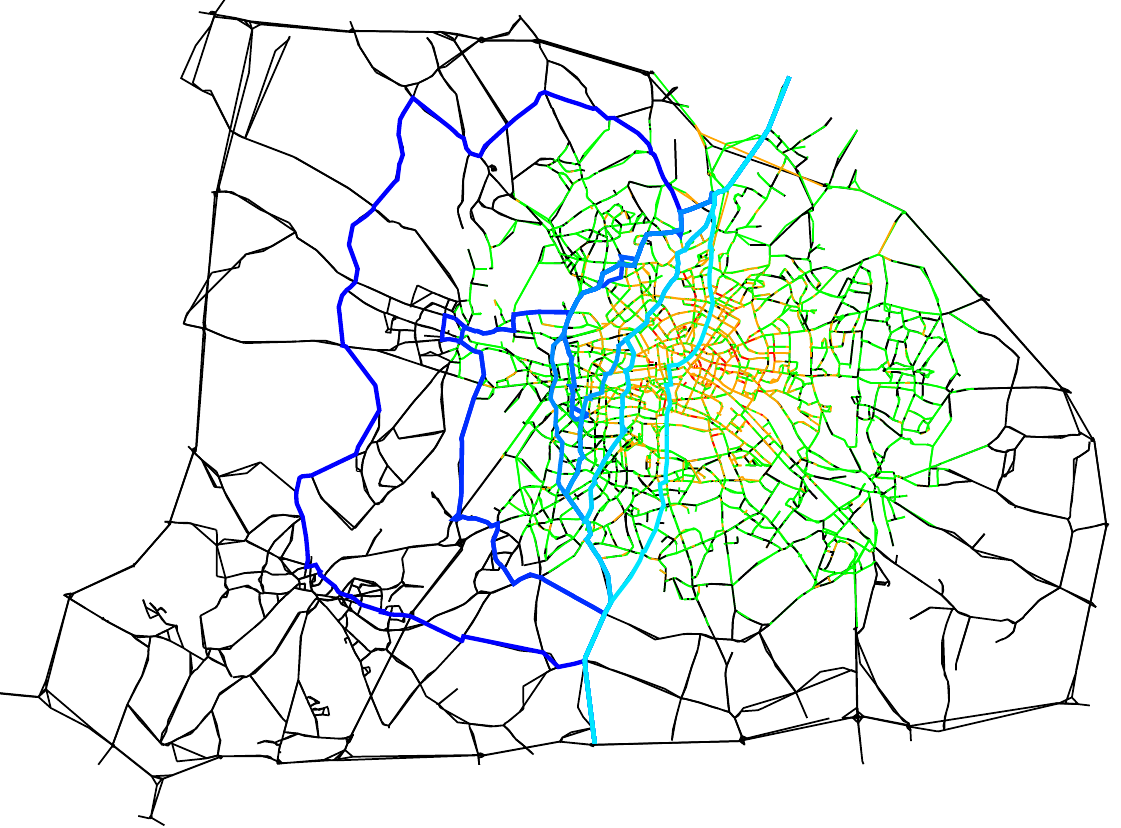}
\end{center}
\caption{Berlin case study - Solution of the variable-sized robust shortest path problem. The color of each edge indicates the degree of uncertainty that affects this edge: Black - almost no uncertainty, green - small uncertainty, orange - medium uncertainty, and red - high uncertainty. The $11$ different paths found as solution of the robust problem are drawn in blue: Light blue represents the nominal solution and dark blue the most robust solution.}
\label{fig_berlin}
\end{figure}

The start node is placed in the north of Berlin and the target node in the south. The nominal solution ignores that the center of Berlin is affected by high uncertainty and goes right through it. The most robust solution avoids all green edges, i.e., edges that are affected by small uncertainty, and takes a long detour around the center of Berlin. Beside these two extreme solutions $9$ compromise solutions are found which balance uncertainty and nominal travel time. 

In the robustness chart (see Figure~\ref{rob_chart}) it is shown how the nominal cost ($\hat{c}^tx$) of the different compromise solution increase with an increasing level of uncertainty ($\lambda$). The chart provides the decision maker with detailed information about the cost of considering larger levels of uncertainty, and shows for how long solutions remain optimal.

\begin{figure}
\begin{center}
\includegraphics[trim = 15mm 15mm 15mm 150mm, clip,width=\textwidth]{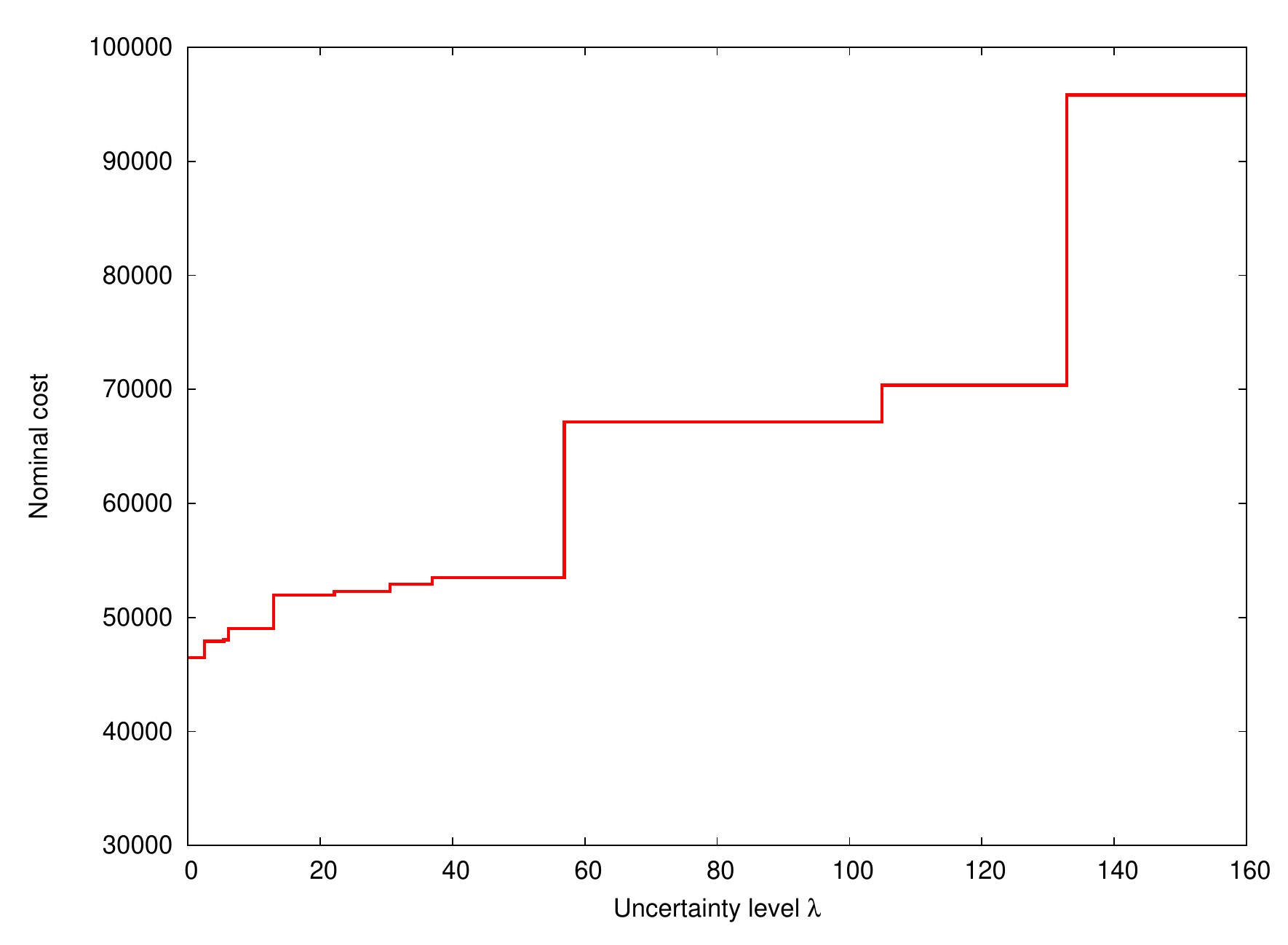}
\end{center}
\caption{Berlin case study - Robustness chart. The curve indicates the nominal cost of the robust solution for some fixed uncertainty level $\lambda$.}
\label{rob_chart}
\end{figure}

We find that variable-sized robust optimization gives a more thorough assessment of an uncertain optimization problem than classic robust approaches would be able to do, while also remaining suitable with regards to computational effort.

\section{Inverse Min-Max Regret Robustness}
\label{sec:inv1}

\subsection{General Discussion}

Having considered variable-sized approaches in the previous section, we now consider the special case of analyzing the nominal solution only.

Let $\hat{x}$ be an optimal solution to the nominal problem $(\mathsf{P})$ with costs $\hat{c}$, i.e.,
\[ \hat{x} \in \arg\min \{ \hat{c}^tx : x\in \X\} \]
Note that in this case, $\hat{x}$ is also an optimizer of the min-max and min-max regret counterparts for the singleton set $\hat{\cU}=\{\hat{c}\}$, i.e., 
\[ \hat{x} \in \arg\min \left\{ \max_{c\in\hat{\cU}} c^tx : x\in\X \right\} \text{ and } \hat{x}\in\arg\min \left\{reg(x,\hat{\cU}) : x\in\X \right\}. \]

In this section on inverse robustness, we analyze for which larger uncertainty sets these properties still hold. We focus on the min-max regret setting, as the more general variable-sized approach covers min-max counterparts in Section~\ref{sec:inv2}. Several ways to do so need to be differentiated.

Firstly, the ``size of an uncertainty set'' needs to be specified. In this section, we discuss two approaches:
\begin{itemize}

\item For uncertainty sets of the form
\[ \cU_\lambda = \left\{ c : c_i \in \left[(1-\ep)\hat{c}_i, (1+\ep)\hat{c}_i\right] \right\},\]
the parameter $\ep\in[0,1]$ specifies the uncertainty size, with $\cU(0) = \hat{\cU}$. We call these sets {\it regular interval sets}. Note that this corresponds to proportional growth in the previous section.

\item For {\it general interval uncertainty sets} 
\[\cU = \{ c : c_i \in [\hat{c}_i-d^-_i,\hat{c}_i+d^+_i] \ \forall i\in[n]\},\]
the size of $\cU$ is the length of intervals $|\cU|:=\sum_{i\in[n]} d^-_i + d^+_i$, with $|\hat{\cU}| = 0$.

\end{itemize}

Secondly, for general uncertainty sets, there are different ways to distribute a fixed amount of slack $d^-,d^+$, all resulting in the same uncertainty size. One can either look for the uncertainty set of the smallest possible size for which $\hat{x}$ is not optimal for the resulting robust objective function; or one can look for the largest possible uncertainty set for which $\hat{x}$ is still optimal in the robust sense. We refer to these approaches as worst-case or best-case inverse robustness, respectively.

Note that $reg(x,\ep) := reg(x,\cU(\ep))$ is a monotonically increasing function in $\ep$ for all $x$. The following calculations show that $reg(x,\ep)$ is a piecewise-linear function in $\ep$. 
\begin{align*}
reg(x,\ep) &= \max_{c \in \cU(\ep)} \left( c^tx - \min_{y \in \X} c^ty \right) \\
&= \max_{y \in \X} \max_{c \in \cU(\ep)} c^tx - c^ty \\
&= \max_{y \in \X}  c(x,\ep)^tx - c(x,\ep)^ty \\
&= \max_{y \in \X}  (1+\ep)c^tx - \sum_i c_i(1-\ep +2\ep x_i) y_i
\end{align*}
Note the $c\in \cU(\ep)$ maximizing $reg(x,\ep)$ is given by $c(x,\ep)_i = c_i(1-\ep +2\ep x_i)$. Observe that $c(x,\ep)_ix_i = (1+\ep)c_ix_i$.

This gives rise to the question: Is there also any difference between worst-case and best-case inverse robustness for regular interval sets? That is, is it possible that a situation occurs as shown in Figure~\ref{fig:plot}, where the regret of the nominal solution becomes larger than the regret of another solution for some $\ep_1$, but for another $\ep_2 > \ep_1$, this situation is again reversed? Here, the best-case robustness approach would yield $\ep=\ep_2$ as the largest value of $\ep$ for which $\hat{x}$ is an optimal regret solution. However, the worst-case approach would give $\ep_1$ as the smallest value for which $\hat{x}$ is not regret optimal.
\begin{figure}[htbp]
\begin{center}
\includegraphics[width=.5\textwidth]{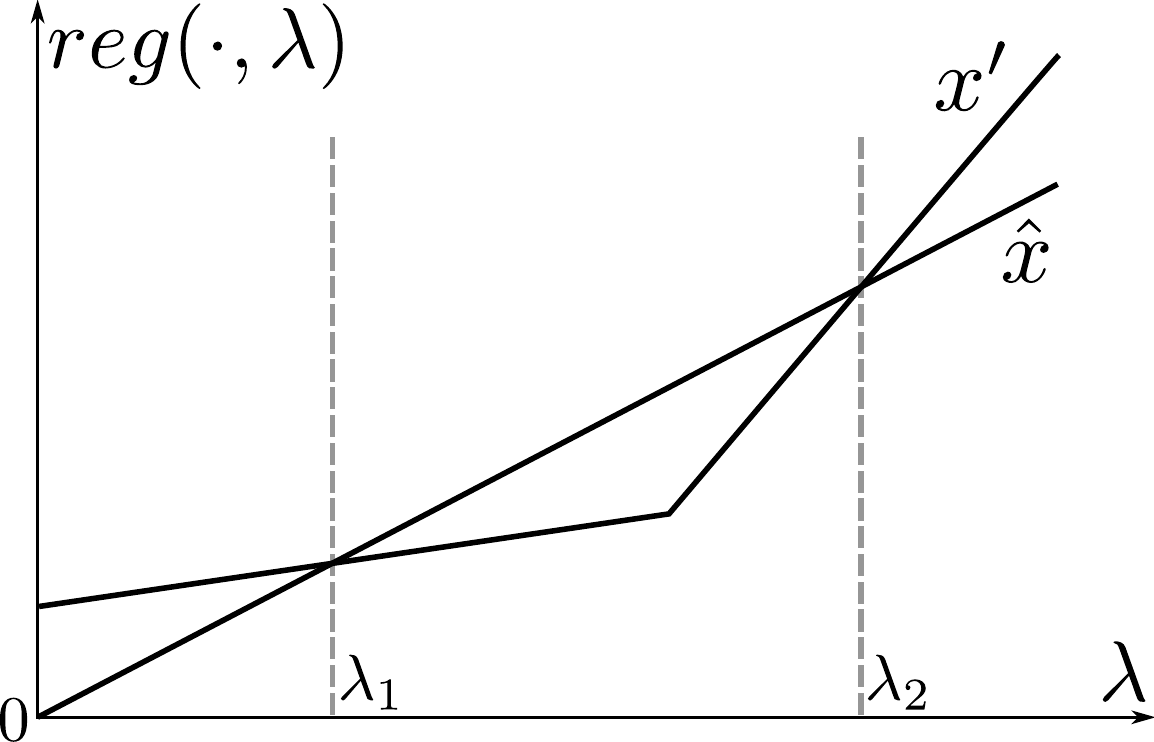}
\caption{Comparison of the regret of two solutions $\hat{x}$ and $x'$ for different values of $\ep$.}\label{fig:plot}
\end{center}
\end{figure}
Such a situation can indeed occur, as the following example demonstrates.

\begin{example}
Consider the shortest path instance from node $1$ to node $6$ in Figure~\ref{fig:path}. We describe paths by the succession of nodes they visit.
\begin{figure}[htbp]
\begin{center}
\includegraphics[width=.8\textwidth]{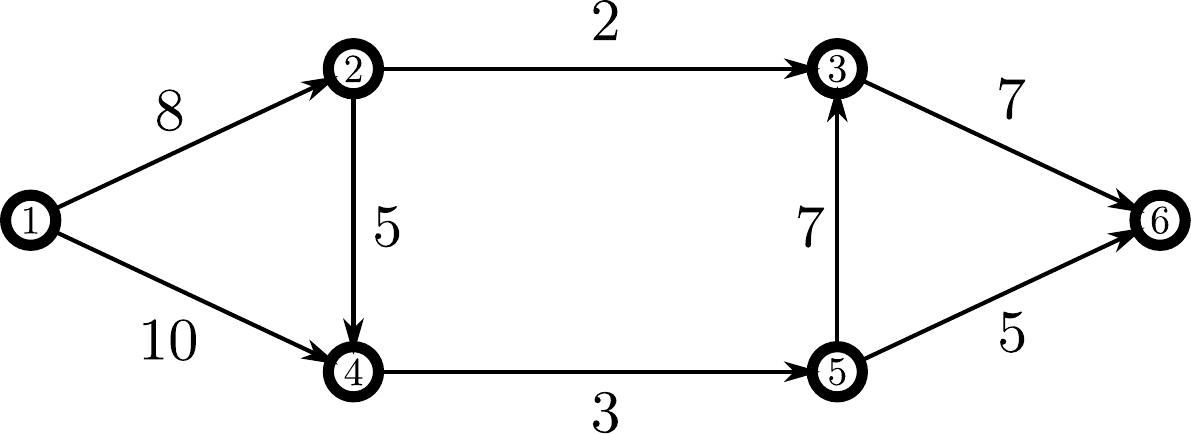}
\caption{Example shortest path instance.}\label{fig:path}
\end{center}
\end{figure}

There are five paths in the graph, $P_1 = (1,2,3,6)$, $P_2 = (1,2,4,5,6)$, $P_3 = (1,2,4,5,3,6)$, $P_4 = (1,4,5,3,6)$ and $P_5 = (1,4,5,6)$. The regret of each path, depending on $\ep$, is shown in Table~\ref{epsex}.

\begin{table}
\begin{center}
\begin{tabular}{r|rrrrrrrrrrr}
$\ep$ & 0.0 & 0.1 & 0.2 & 0.3 & 0.4 & 0.5 & 0.6 & 0.7 & 0.8 & 0.9 & 1.0 \\
\hline
$reg(P_1,\ep)$ & \textbf{0.0} & 2.5 & 6.0 & 9.5 & 13.0 & \textbf{16.5} & 20.0 & 23.5 & 27.0 & 30.5 & \textbf{34.0} \\
$reg(P_2,\ep)$ & \textbf{4.0} & 6.2 & 8.4 & 10.6 & 12.8 & \textbf{15.0} & 19.2 & 23.4 & 27.6 & 31.8 & \textbf{36.0} \\
$reg(P_3,\ep)$ & 13.0 & 16.2 & 20.4 & 24.6 & 28.8 & 33.0 & 37.2 & 41.4 & 45.6 & 49.8 & 54.0 \\
$reg(P_4,\ep)$ & 10.0 & 13.0 & 16.0 & 19.0 & 22.8 & 27.0 & 31.2 & 35.4 & 39.6 & 43.8 & 48.0 \\
$reg(P_5,\ep)$ & 1.0 & 4.5 & 8.0 & 11.5 & 15.0 & 18.5 & 22.0 & 25.5 & 29.0 & 32.5 & 36.0
\end{tabular}
\caption{Regret of paths in Figure~\ref{fig:path} depending on $\ep$.}\label{epsex}
\end{center}
\end{table}

A unique minimizer of the nominal scenario (i.e., $\ep = 0$) is the path $P_1$ with a regret of zero, while path $P_2$ has a regret of four. For $\ep=0.5$, the regret of path $P_1$ becomes $16.5$, but is only $15$ for path $P_2$. Finally, for $\ep=1$, path $P_1$ has a regret of $34$, but path $P_2$ has a regret of $36$, i.e., path $P_2$ has the better regret for $\ep=0.5$, but not for $\ep=0$ and $\ep=1$. Also, $P_1$ is indeed the path with the smallest regret for $\ep = 1$.
\end{example}

Note how such an example is in contrast to variable-sized min-max robust optimization, where a solution that ceases to be optimal for some value $\ep$ will not be optimal again for another value $\ep' > \ep$.

On a final note, we only focus on analyzing the nominal solution in the following. However, the models we present can be extended to find a solution to the whole variable-sized robust problem. To this end, a solution $x^*$ that is optimal for some $\ep^* > 0$ can be analyzed by imposing lower bounds of the form $\ep \ge \ep^*$.

In Section~\ref{sec:reg}, we discuss regular interval uncertainty sets, where we specifically analyze the best-case setting in Section~\ref{sec:epBC} and the worst-case setting in Section~\ref{sec:epWC}. Section~\ref{sec:gen} then presents results on general interval sets.

\subsection{Regular Interval Uncertainty Sets}
\label{sec:reg}

\subsubsection{Best-Case Inverse Robustness}
\label{sec:epBC}

The best-case inverse problem we consider here can be summarized as: Given some solution $\hat{x}$, what is the largest amount of uncertainty that can be added such that $\hat{x}$ is still optimal for the resulting min-max regret problem? 

More formally, the best-case inverse problem we consider here is given as:
\begin{align*}
\max\ & \ep \\
\text{s.t. } & reg(\hat{x},\ep) \le reg(\tilde{x},\ep) & \forall \tilde{x}\in\X \\
& \ep \in [0,1]
\end{align*}
where $reg(x,\ep) = \max_{c\in\cU_\ep} \left( c^t x - opt(c)\right)$.

We now focus on combinatorial problems, where we assume that the strong duality property holds, i.e., when can solve the continuous relaxation of the problem to find an optimal solution. Example problems where this is the case include the assignment problem, the shortest path problem, or the minimum spanning tree problem. Strong duality is an important tool in classic min-max regret problems, as it allows a compact problem formulation \cite{Aissi2009}. As an example, we consider the assignment problem in the following. The nominal problem is given by
\begin{align*}
\min\ & \sum_{i\in[n]} \sum_{j\in[n]} \hat{c}_{ij} x_{ij} \\
\text{s.t. } & \sum_{i\in[n]} x_{ij} = 1 & \forall j\in[n] \\
& \sum_{j\in[n]} x_{ij} = 1 & \forall i\in[n] \\
& x_{ij} \in \{0,1\} & \forall i,j\in [n]
\end{align*}
Using that $reg(x,\ep) = c^t(x) x- opt(c(x))$ with
\[ c_{ij}(x) = \begin{cases} (1+\lambda) \hat{c}_{ij}& \text{ if } x_{ij} = 1,\\
 (1-\lambda) \hat{c}_{ij} & \text{ else}\end{cases}\]
and dualizing the inner optimization problem, we find a compact formulation of the min-max regret problem as follows:
\begin{align*}
\min\ & \sum_{i\in[n]} \sum_{j\in[n]} (1+\ep) \hat{c}_{ij} x_{ij} - \sum_{i\in[n]} (u_i + v_i) \\
\text{s.t. } & u_i + v_j \le (1-\ep) \hat{c}_{ij} + 2\ep\hat{c}x_{ij} & \forall i,j\in[n]\\
& \sum_{i\in[n]} x_{ij} = 1 & \forall j\in[n] \\
& \sum_{j\in[n]} x_{ij} = 1 & \forall i\in[n] \\
& x_{ij} \in \{0,1\} & \forall i,j\in [n] \\
& u_i, v_i \gtrless 0 & \forall i\in [n]
\end{align*}
We now re-consider the inverse problem. Note that to reformulate the constraints
\[ reg(\hat{x},\ep) \le reg(\tilde{x},\ep) \qquad \forall \tilde{x}\in\X \]
we can use the same duality approach for the left-hand side (as an optimal solution aims at having this side as small as possible), but not for the right-hand side (which should be as large as possible). Instead, for each $\tilde{x}\in\X$, we need to provide a primal solution. Enumerating all possible solutions in $\X$ as $\tilde{x}^k$, $k\in [|\X|]$, the inverse problem can hence be reformulated as
\begin{align}
\max\ & \ep \label{eq-bc1}\\
\text{s.t.} & \sum_{i,j\in[n]} (1+\ep) \hat{c}_{ij} \hat{x}_{ij} - \sum_{i\in[n]} (u_i+v_i) \nonumber\\
& \hspace{1cm} \le \sum_{i,j\in [n]} (1+\ep)\hat{c}_{ij} \tilde{x}^k_{ij} - \sum_{i,j\in[n]} ((1-\ep) \hat{c}_{ij} + 2\ep\hat{c}_{ij}\tilde{x}^k_{ij}) x^k_{ij} 
& \forall k \in [|\X|] \label{hardcon}\\
& u_i + v_j \le (1-\ep)\hat{c}_{ij} + 2\ep\hat{c}_{ij} \hat{x}_{ij} & \forall i,j\in[n] \label{eq-bc2}\\
& \ep \in [0,1] \\
& u_i,v_i \gtrless 0 & \forall i\in[n] \label{eq-bc4}\\
& x^k \in \X & \forall k \in [|\X|] \label{eq-bc5}
\end{align}
Here, the objective function~\eqref{eq-bc1} is to maximize the size of the uncertainty set $\ep$. Constraints~\eqref{hardcon} model that the regret of $\hat{x}$ needs to be at most as large as the regret of all possible alternative solutions $\tilde{x}^k$. To calculate the regret of these solutions, additional solutions $x^k$ are required. The duality constraints in \eqref{eq-bc2} ensure that the regret of $\hat{x}$ is calculated correctly.

To solve this problem, not all variables and constraints need to be included from the beginning. Instead, we can generate them during the solution process of the master problem by solving subproblems that aim at minimizing the right-hand side of Constraint~\eqref{hardcon}. This is a classic min-max regret problem again.

To resolve the non-linearity between $x^k$ and $\ep$, we use additional variables $y^k_{ij} := \ep x^k_{ij}$. The resulting mixed-integer program is then given as:
\begin{align}
\max\ & \ep\\
\text{s.t.} & \sum_{i,j\in[n]} (1+\ep) \hat{c}_{ij} \hat{x}_{ij} - \sum_{i\in[n]} (u_i+v_i) \\
& \hspace{1cm} \le \sum_{i,j\in [n]} (1+\ep)\hat{c}_{ij} \tilde{x}^k_{ij} - \sum_{i,j\in[n]} (\hat{c}_{ij} x^k_{ij} + (2\tilde{x}^k_{ij} -1) \hat{c}_{ij} y^k_{ij} )
\hspace{-1cm}& \forall k \in [|\X|] \label{ep-bc-con1}\\
& u_i + v_j \le (1-\ep)\hat{c}_{ij} + 2\ep\hat{c}_{ij} \hat{x}_{ij} & \forall i,j\in[n] \label{ep-bc-con2}\\
& 0 \le y^k_{ij} \le \ep & \forall i,j\in[n], k\in [|\X|] \\
& \ep + x^k_{ij} - 1 \le y^k_{ij} \le x^k_{ij} & \forall i,j\in[n], k\in [|\X|] \\
& \ep \in [0,1] \\
& u_i,v_i \gtrless 0 & \forall i\in[n]\\
& y^k_{ij} \ge 0 & \forall i,j\in[n], k\in [|\X|] \\
& x^k \in \X & \forall k \in [|\X|] 
\end{align}

For a more general formulation, which is not restricted to the assignment problem, let us assume that $\X = \{ x: Ax \ge b, x\in\{0,1\}^n\}$ with $A\in\R^{m\times n}$ and $b\in\R^m$, and that strong duality holds. 
Then the min-max regret problem with interval sets $\cU_\ep$ can be formulated as
\[ \min_{x\in\X,u\in\Y} \left( (1+\ep)\hat{c} x - b^tu \right) \]
with $\Y = \{ u: A^tu \le c(x) , u\ge 0\}$, see \cite{Aissi2009}. Using this compact formulation for min-max regret, we can substitute Constraint~\eqref{ep-bc-con1} for
\[ \sum_{i\in[n]} (1+ \ep)\hat{c}_i \hat{x}_i - b^t u \le \sum_{i\in [n]} (1+\ep)\hat{c}_i \tilde{x}^k_i - \sum_{i\in[n]} (\hat{c}_i x^k_i + (2\tilde{x}^k_i -1) \hat{c}_i y^k_i ) \qquad \forall k \in[|\X|] \]
and Constraints~\eqref{ep-bc-con2} for 
\[ (A^tu)_i \le (1-\ep) \hat{c}_i + 2\ep\hat{c}_i \hat{x}_i \qquad \forall i\in [n] \]
with dual variables $u_i\ge 0$ for all $i\in[m]$.

We conclude this section by briefly considering the case where the original problem $(\mathsf{P})$ does not have zero duality gap. In this case, we rewrite the constraints
\[ reg(\hat{x},\ep) \le reg(\tilde{x}^k,\ep) \qquad \forall k\in[|\X|] \]
as
\begin{align*}
&\sum_{i\in[n]} ( 1+\ep) \hat{c}_i \hat{x}_i - \sum_{i\in[n]} \left( (1-\ep)\hat{c}_i + 2\ep\hat{c}_i \hat{x}_i \right)\bar{x}^\ell_i \\
\quad & \le \sum_{i\in [n]} (1+\ep)\hat{c}_i  \tilde{x}^k_i - \sum_{i\in[n]} ( (1-\ep)\hat{c}_i +2\ep\hat{c}_i \tilde{x}^k_i) x^k_i & \forall k,\ell \in [|\X|]
\end{align*}
i.e., we compute the regret on both sides of the inequality by using all primal solutions as comparison. As before, variables and constraints can be generated iteratively during the solution process. To find the next solution $\bar{x}^\ell$, only a problem $(\mathsf{P})$ or the original type needs to be solved.

\subsubsection{Worst-Case Inverse Robustness}
\label{sec:epWC}

We now consider the worst-case inverse problem, which may be summarized as: Given some solution $\hat{x}$, what is the smallest amount of uncertainty that needs to be added such that $\hat{x}$ is not optimal for the resulting min-max regret problem anymore?

More formally, the problem we consider can be denoted as:
\begin{align}
\min\ & \ep \\
\text{s.t. } & reg(\hat{x},\ep) \ge reg(\tilde{x},\ep) + \varepsilon  \label{ep-wc1}\\
& \ep\in[0,1] \\
& \tilde{x}\in\X
\end{align}
where $\varepsilon$ is a small constant, i.e., we need to find an uncertainty parameter $\ep$ and an alternative solution $\tilde{x}$ such that the regret of $\tilde{x}$ is at least better by $\varepsilon$ than the regret of $\hat{x}$.

Note that if $\hat{x}$ is not a unique minimizer of the nominal scenario, there is an uncertainty set $\cU$ with arbitrary small size such that $\hat{x}$ is not the minimizer of the regret anymore. It suffices to increase an element of $\hat{x}$, which is not included in another minimizer of the nominal scenario. Hence, this approach is most relevant for unique minimizers of the nominal scenario.

As in the previous section, we use the assignment problem as an example how to rewrite this problem in compact form. As an optimal solution will aim at having the right-hand side of Constraint~\eqref{wc1} as small as possible, we use strong duality to write:
\begin{align*}
\min\ & \ep \\
\text{s.t. } & reg(\hat{x},\ep) \ge \sum_{i,j\in[n]} (1+\ep)\hat{c}_{ij} \tilde{x}_{ij} - \sum_{i\in[n]} (u_i+v_i) + \varepsilon  \\
& u_i + v_j \le (1-\ep) \hat{c}_{ij} + 2\ep\hat{c}_{ij} \tilde{x}_{ij} & \forall i,j\in[n] \\
& \ep\in[0,1] \\
& \tilde{x}\in\X
\end{align*}
For the left-hand side of Constraint~\eqref{ep-wc1}, we need to include an additional primal solution $x'$ as an optimal solution for the worst-case scenario of $\hat{x}$. Linearizing the resulting products by setting $y_{ij}:= \ep x'_{ij}$ and $\beta_{ij}:= \ep \tilde{x}_{ij}$, the resulting problem formulation is then:
\begin{align}
\min\ & \ep \\
\text{s.t. } & 
\sum_{i,j\in [n]} (1+\ep) \hat{c}_{ij} \hat{x}_{ij} - \sum_{i,j\in[n]} ( \hat{c}_{ij} x'_{ij} + (2\hat{x}_{ij} - 1) \hat{c}_{ij} y_{ij} )\\
& \hspace{1cm} \ge \sum_{i,j\in[n]} (\hat{c}_{ij}\tilde{x}_{ij}+\hat{c}_{ij} \beta_{ij})- \sum_{i\in[n]} (u_i+v_i) + \varepsilon  \label{ep-wc-con1}\\
& u_i + v_j \le (1-\ep) \hat{c}_{ij} + 2\hat{c}_{ij} \beta_{ij}  & \forall i,j\in [n] \label{ep-wc-con2} \\
& 0 \le y_{ij} \le \ep & \forall i,j\in[n]\\
& \ep + x'_{ij} -1 \le y_{ij} \le \ep & \forall i,j\in[n]\\
& 0 \le \beta_{ij} \le \ep & \forall i,j\in[n]\\
& \ep + \tilde{x}_{ij} -1 \le \beta_{ij} \le \ep & \forall i,j\in[n]\\
& \sum_{i\in[n]} x'_{ij} = 1 & \forall i\in[n] \\
& \sum_{j\in[n]} x'_{ij} = 1 & \forall j\in[n] \\
& \sum_{i\in[n]} \tilde{x}_{ij} = 1 & \forall i\in[n] \\
& \sum_{j\in[n]} \tilde{x}_{ij} = 1 & \forall j\in[n] \\
& \ep \in[0,1] \\
& x'_{ij},\tilde{x}_{ij} \in \{0,1\} & \forall i,j\in[n] \\
& y_{ij},\beta_{ij} \in [0,1] & \forall i,j\in[n]
\end{align}
Note that this formulation has polynomially many constraints and variables, and thus can be attempted to be solved without an iterative procedure. In general, for a problem with $\X = \{ x: Ax \ge b, x\in\{0,1\}^n\}$ and the strong duality property, we can substitute Constraint~\eqref{ep-wc-con1} for
\[ \sum_{i\in[n]} (1+\ep)\hat{c}_i \hat{x}_i - \sum_{i\in[n]} ( \hat{c}_i x'_i + (2\hat{x}_i -1)\hat{c}_iy_i ) \ge \sum_{i\in[n]} (\hat{c}_i\tilde{x}_i+ \hat{c}_i\beta_i)- b^t u + \varepsilon \]
and Constraints~\eqref{ep-wc-con2} for 
\[ (A^tu)_i \le (1-\ep)\hat{c}_i + 2\hat{c}_i\beta_i \qquad \forall i\in [n] \]
with dual variables $u_i \ge 0$ for all $i\in[m]$.

In the case that problem $(\mathsf{P})$ does not have a zero duality gap, we propose to follow a similar strategy as in Section~\ref{sec:epBC}. That is, we consider constraints
\begin{align*}
& \sum_{i\in[n]} (1+\ep)\hat{c}_i \hat{x}_i - \sum_{i\in[n]} ( \hat{c}_i x'_i + (2\hat{x}_i -1)\hat{c}_iy_i ) \\
& \hspace{1cm} \ge \sum_{i\in[n]} (1+\ep) \hat{c}_i \tilde{x}_i - \sum_{i\in[n]} ((1-\ep)\hat{c}_i + 2\ep\hat{c}_i\tilde{x}_i ) x^k_i  + \varepsilon & \forall k \in [|\X|] 
\end{align*}
and generate them iteratively in a constraint relaxation procedure. Note that no additional variables are required.

\subsection{General Interval Uncertainty Sets}
\label{sec:gen}

\subsubsection{Best-Case Inverse Robustness}
\label{sec:genBC}

We now consider consider general interval uncertainty, where the size of the uncertainty is the summed length of intervals. An example problem is discussed in the following.
\begin{example}\label{ex1}
An assignment instance with the corresponding best-case inverse robustness solution is given in Figure~\ref{fig-bc}. In this example, we only allow costs to increase on edges in the nominal solution, and only to decrease on all other edges. On the left side, the nominal solution $\hat{x} = \{ a, f, h\}$ is shown in bold along with an edge labeling. The nominal costs are $\hat{c} = \{4,1,1,7,5,4,8,4,8\}$. On the right side, a largest possible uncertainty set $\cU\subset \R^{n\times n}_+$ that preserves optimality of $\hat{x}$ is given. We have $|\cU|=29$.
\begin{figure}[htbp]
\begin{center}
\includegraphics[width=0.9\textwidth]{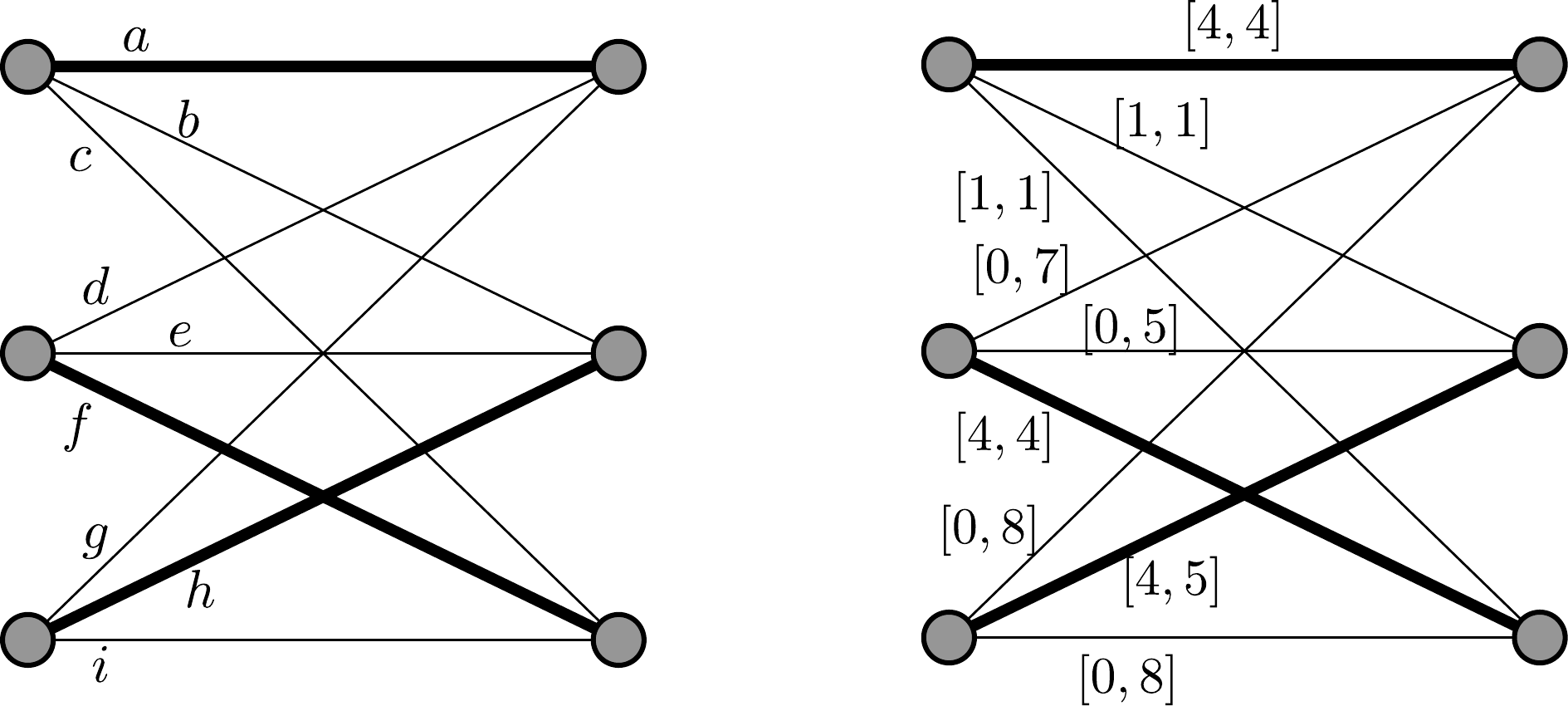}
\end{center}
\caption{Nominal solution left, largest uncertainty set right.}\label{fig-bc}
\end{figure}
Note that the uncertainty is chosen in a way such that all other possible solutions have a regret at least as large as the regret of $\hat{x}$, see Table~\ref{tab:ex1}. Whenever we increase the costs of an edge in $\{a,f,h\}$, the values in the corresponding rows are increased. If we decrease the costs of an edge in $\{b,c,d,e,g,i\}$, then the values in the corresponding column are increased. Table~\ref{tab:ex1} shows that there is no way to further increase or decrease costs without losing optimality of $\hat{x}$.

\begin{table}[htb]
\begin{center}
\begin{tabular}{r|r|r|r|r|r|r}
& $\{a,e,i\}$ & $\{a,f,h\}$ & $\{b,d,i\}$ & $\{b,f,g\}$ & $\{c,d,h\}$ & $\{c,e,g\}$ \\
\hline
$\{a,e,i\}$ & 0 & 5 & 8 & \textbf{12} & \textbf{12} & 3 \\
$\hat{x} = \{a,f,h\}$ & 9 & 0 & \textbf{12} & 8 & 7 & \textbf{12} \\
$\{b,d,i\}$ & 4 & 4 & 0 & 11 & 4 & \textbf{15} \\
$\{b,f,g\}$ & 9 & 11 & \textbf{12} & 0 & 8 & 4 \\
$\{c,d,h\}$ & 9 & 0 & 5 & 8 & 0 & \textbf{12} \\
$\{c,e,g\}$ & 5 & 2 & \textbf{13} & 1 & 9 & 0
\end{tabular}
\caption{Regret of all solutions for the example in Figure~\ref{fig-bc}. Rows correspond to the solution for which the regret is calculated, and columns stand for the possible solution of the inner optimization problem, i.e., the largest value per row is the regret of the respective solution.}\label{tab:ex1}
\end{center}
\end{table}

\end{example}

More formally, the best-case inverse problem we consider here is given as:
\begin{align*}
\max\ & \sum_{i\in[n]} d^+_i + d^-_i \\
\text{s.t. } & reg(\hat{x},d^+,d^-) \le reg(\tilde{x},d^+,d^-) & \forall \tilde{x}\in\X \\
& d^+_i \in [0,M^+_i] & \forall i\in[n]\\
& d^-_i \in [0,M^-_i] & \forall i\in[n]
\end{align*}
where $M^+_i$, $M^-_i$ denote the maximum possible deviations in each coefficient, and
\[ reg(x,d^+,d^-) = \max_{c\in\cU(d^+,d^-)} c^tx - opt(c). \]
By setting $M^+_i = 0$ or $M^-_i = 0$ for some index $i$, we can model that this coefficient may not deviate in the respective direction, as in Example~\ref{ex1}.

We first consider this setting with the unconstrained combinatorial optimization problem, where $\X = \{0,1\}^n$.

Let $\hat{x}$ be optimal for $\hat{c}$. Note that for some fixed $c$, an optimal solution is to pack all items with negative costs. Therefore, we assume
\[ \hat{x}_i = \begin{cases} 1 & \text{ if } \hat{c}_i \le 0 \\
0 & \text{ else } \end{cases} \]
There are no other optimal solutions, except for indices where $\hat{c}_i = 0$.
To describe optimal solutions for the min-max regret problem under uncertainty, we make use of the following lemma:
\begin{lemma}{(See \cite{regretellipsen})}
Let $\cU = \times_{i\in[n]} [\hat{c}_i - d_i, \hat{c}_i + d_i]$ for the unconstrained combinatorial optimization problem. Then, an optimal solution for $\hat{c}$ is also an optimal solution for the min-max regret problem.
\end{lemma}

In our setting, this becomes:
\begin{lemma}\label{uclem}
Let $\cU = \cU(d^+,d^-)$. Then, $x^*$ with
\[ x^* = \begin{cases} 1 & \text{ if } 2\hat{c}_i + d^+_i - d^-_i \le 0 \\ 0 & \text{ else} \end{cases} \]
is an optimal solution for the min-max regret problem.
\end{lemma}
Note that there are no other optimal solutions, except for indices where $2\hat{c}_i + d^+_i - d^-_i = 0$. We can therefore describe the largest possible uncertainty set such that $\hat{x}$ remains optimal for the min-max regret problem in the following way:
\begin{theorem}
Let an unconstrained problem with cost $\hat{c}$ be given. The largest uncertainty set of the form $\cU(d^+,d^-)$ such that $\hat{x}$ remains optimal for the resulting regret problem is given by $d^+$ and $d^-$ with the following properties:
\begin{itemize}
\item If $\hat{c}_i \le 0$, then
\begin{align*}
d^+_i &= \min\{ M^-_i - 2\hat{c}_i, M^+_i\}  \\
d^-_i &= M^-_i
\end{align*}
\item If $\hat{c}_i > 0$, then
\begin{align*}
d^+_i &= M^+_i \\
d^-_i &= \min\{M^+_i + 2\hat{c}, M^-_i\}
\end{align*}
\end{itemize}
\end{theorem}
\begin{proof}
Let $\hat{c}_i \le 0$ and $\hat{x}_i = 1$. Using Lemma~\ref{uclem}, we choose $d^+_i$ and $d^-_i$ such that $2\hat{c}_i \le d^-_i - d^+_i$. Setting $d^-_i = M^-_i$ and solving for $d^+_i$, we find $d^+_i = \min\{ M^-_i - 2\hat{c}_i, M^+_i\}$. Analogously for $\hat{c}_i > 0$ and $\hat{x}_i = 0$.
\end{proof}

\begin{corollary}
For $M^-_i = 0$ and $M^+_i = \infty$ when $\hat{x}_i = 1$, and $M^-_i = \infty$ and $M^+_i = 0$ when $\hat{x}_i = 0$, we have 
\[ d^+_i = -2\hat{c}_i \text{ for } \hat{c}_i < 0 \text{ and } d^-_i = 2\hat{c}_i \text{ for } \hat{c}_i \ge 0 \]
and the corresponding uncertainty size is hence $|\cU(d^+,d^-)| = \sum_{i\in[n]} 2|\hat{c}_i|$.
\end{corollary}

For general combinatorial problems with the strong duality property, we can follow a similar reformulation procedure as described in Section~\ref{sec:epBC}. For the sake of brevity, we only give the final, linearized formulation for $\X = \{ x: Ax \ge b, x\in\{0,1\}^n\}$ here:
\begin{align*}
\max\ & \sum_{i\in[n]} d^+_i + d^-_i \\
\text{s.t.} & \sum_{i\in[n]} (\hat{c}_i+d^+_i) \hat{x}_i - b^t u\\
& \hspace{.5cm}\le\sum_{i\in [n]} (\hat{c}_i + d^+_i) \tilde{x}^k_i - 
\sum_{i\in[n]} ( \hat{c}_i x^k_i - z^k_i + \tilde{x}^k_iy^k_i  + \tilde{x}^k_i z^k_i ) \hspace{-1.5cm} & \forall k\in[|\X|] \\
& (A^tu)_i\le \hat{c}_i - d^-_i + (d^+_i + d^-_i) \hat{x}_i & \forall i\in[n] \\
& 0 \le y^k_i \le d^+_i & \forall i\in[n],\forall k\in[|\X|]\\
& d^+_i - M^+_i (1-x^k_i) \le y^k_i \le M^+_i x^k_i & \forall i\in[n],\forall k\in[|\X|]\\
& 0 \le z^k_i \le d^-_i & \forall i\in[n],\forall k\in[|\X|]\\
& d^-_i -M^-_i (1-x^k_i) \le z^k_i \le M^-_i x^k_i & \forall i\in[n],\forall k\in[|\X|]\\
& Ax^k \ge b & \forall k\in[|\X|]\\
& d^+_i \in [0,M^+_i] & \forall i\in[n] \\
& d^-_i \in [0,M^-_i] & \forall i\in[n] \\
& u_i \ge 0 & \forall i\in[m] \\
& y^k_i \in [0,M^+_i]& \forall i\in[n],\forall k\in[|\X|] \\ 
& z^k_i \in [0,M^-_i] & \forall i\in[n],\forall k\in[|\X|] \\
& x^k_i \in \{0,1\} & \forall i\in[n],\forall k\in[|\X|]
\end{align*}

\subsubsection{Worst-Case Inverse Robustness}
\label{sec:genWC}

We now consider the worst-case inverse problem with general interval uncertainty, which is illustrated in the following example.

\begin{example}
A small example for such an inverse problem is shown in Figure~\ref{fig:wc}. The nominal assignment problem along with an optimal nominal solution $\hat{x}$ (bold edges) is given in the left part of the figure. In the right part, we show a smallest possible uncertainty set such that $\hat{x}$ is not optimal for the resulting min-max regret problem, with an optimal solution in bold edges. We have $|\cU|=14+\varepsilon$.
\begin{figure}[htbp]
\begin{center}
\includegraphics[width=.9\textwidth]{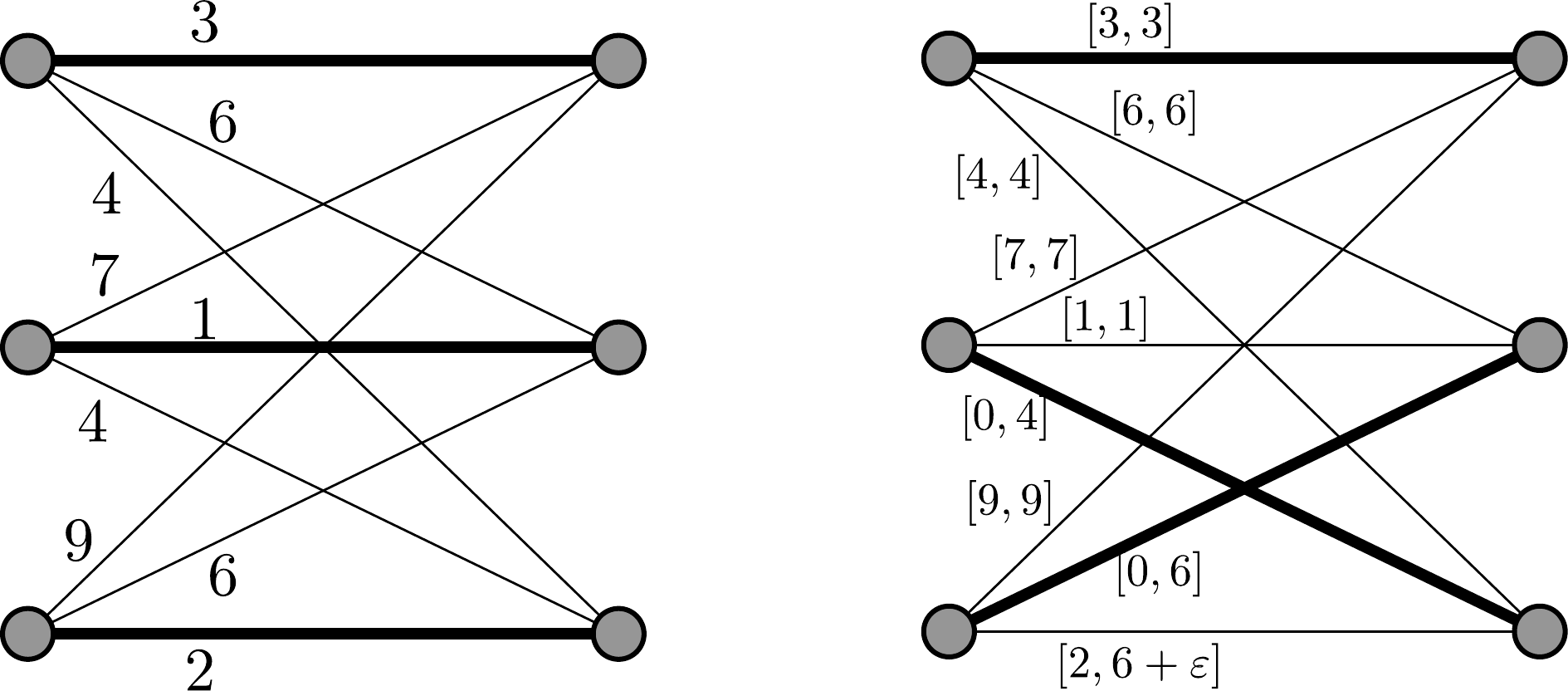}
\caption{Nominal solution left, smallest uncertainty set right.}\label{fig:wc}
\end{center}
\end{figure}
We require $\varepsilon > 0$. Note that for $\varepsilon = 0$, both solutions would have the same regret value $Reg(\hat{x}) = 7$.
\end{example}

More formally, the problem we consider can be denoted as:
\begin{align}
\min\ & \sum_{i\in[n]} d^+_i + d^-_i \\
\text{s.t. } & reg(\hat{x},d^+,d^-) \ge reg(\tilde{x},d^+,d^-) + \varepsilon  \label{wc1}\\
& d^+_i \in [0,M^+_i] & \forall i\in[n] \\
& d^-_i \in [0,M^-_i] & \forall i\in[n] \\
& \tilde{x}\in\X
\end{align}
where $\varepsilon$ is a small constant, i.e., we need to find uncertainty parameters $d^+$, $d^-$ and an alternative solution $\tilde{x}$ such that the regret of $\tilde{x}$ is at least better by $\varepsilon$ than the regret of $\hat{x}$.

As in the previous sections, we use the assignment problem as an example how to rewrite this problem in compact form. The resulting problem formulation for $\X = \{ x: Ax \ge b, x\in\{0,1\}^n\}$ is then:
\begin{align*}
\min\ & \sum_{i\in[n]} d^+_i + d^-_i \\
\text{s.t. } & 
\sum_{i\in [n]} (\hat{c}_i + d^+_i) \hat{x}_i - \sum_{i\in[n]} ( \hat{c}_i x'_i - z_i + \hat{x}_i y_i  + \hat{x}_iz_i ) \\
& \hspace{1cm} \ge \sum_{i\in[n]} (\hat{c}_i\tilde{x}_i+\beta_i) - b^t u + \varepsilon\\
& (A^tu)_i \le \hat{c}_i - d^-_i + \beta_i + \gamma_i & \forall i\in [n]  \\
& 0 \le y_i \le d^+_i & \forall i\in[n]\\
& d^+_i -M^+_i (1-x'_i) \le y_i \le M^+_i x'_i & \forall i\in[n]\\
& 0 \le z_i \le d^-_i & \forall i\in[n]\\
& d^-_i -M^-_i (1-x'_i) \le z_i \le M^-_i x'_i & \forall i\in[n]\\
& 0 \le \beta_i \le d^+_i & \forall i\in[n]\\
& d^+_i -M^+_i (1-\tilde{x}_i) \le \beta_i \le M^+_i \tilde{x}_i & \forall i\in[n]\\
& 0 \le \gamma_i \le d^-_i & \forall i\in[n]\\
& d^-_i -M^-_i (1-\tilde{x}_i) \le \gamma_i \le M^-_i \tilde{x}_i & \forall i\in[n]\\
& Ax' \ge b \\
& A\tilde{x} \ge b \\
& d^+_i,  y'_i, \beta_i \in [0,M^+_i] & \forall i\in[n] \\
& d^-_i,z'_i,\gamma_i  \in [0,M^-_i] & \forall i\in[n]\\
& u_i \ge 0 & \forall i\in[m] \\
& x'_i,\tilde{x}_i \in \{0,1\} & \forall i\in[n]
\end{align*}

\subsection{Computational Insight}
\label{sec:exp}

\subsubsection{Setup}

In this section we consider best-case and worst-case inverse robustness as a way to find structural insight into differences of robust optimization problems.

To this end, we used the following experimental procedure. We generated random assignment instances in complete bipartite graphs of size $15 \times 15$ (i.e., there are 225 edges). For every edge $e$, we generate a random nominal weight $\hat{c}_e$ uniformly in $\{0,\ldots,20\}$. We generated 2,500 instances this way.

For each instance, we solve the best-case and worst-case inverse robustness problems, where we allow symmetric deviations (i.e., $d^+=d^-$) in the interval $[0,20]$. Best-case problems are solved as described in Section~\ref{sec:genBC} using the iterative procedure that constructs additional variables and constraints by solving a min-max regret problem as sub-procedure. Worst-case problems are solved using the compact formulation from Section~\ref{sec:genWC}.

Additionally, for each instance, we create 500 min-max regret problems with randomly generated symmetric interval uncertainty within the same maximum range $[0,20]$ of possible deviations. Each min-max regret instance is solved to optimality using the compact formulation based on dualizing the inner problem. Additionally, we calculate the objective value of the nominal solution for each min-max regret instance.

To solve optimization problems, we used Cplex v.12.6 \cite{cplex} on a computer with a 16-core Intel Xeon E5-2670 processor, running at 2.60 GHz with 20MB cache, and Ubuntu 12.04. Processes were pinned to one core.

\subsubsection{Results and Discussion}

We present key values for worst-case inverse problems in Table~\ref{tab:wc}. We categorized instances according to the objective value ''WC'' of the worst-case problem. The smallest observed objective value was 4, and the largest was 28 (the value 2 could not be achieved, as we required the difference between regret values to be at least 1).

Column ''Freq'' denotes how often an objective value was observed over the 2,500 instances. Columns ''Reg'' and ''NomReg'' show the average optimal regret, and the average regret of the nominal solution within each instance class, respectively. In column ''Ratio'', the ratio between these two values is given. Column ''BC'' shows the average best-case inverse value for problems within each class. The BC values are given as the negative difference to the maximum BC value, which is 9000 (i.e., smaller BC values mean that the largest possible interval uncertainty for which the nominal solution is also the optimal solution for the regret problem is smaller). ''WCT'' and ''BCT'' show the average time to solve the worst-case and the best-case inverse problems in seconds, respectively. ''RegT'' is the average time to solve the 500 min-max regret problems we generated per instance and is also given in seconds. 

\begin{table}[htb]
\centering
\begin{tabular}{r|r|rrr|r|rrr}
WC & Freq & Reg & NomReg & Ratio & BC & WCT & BCT & RegT \\
\hline
 & & & & & & & & \\[-2.2ex]
4 & 569 & 297.014 & 348.673 & 1.1739 & -3.47 & 0.36 & 13.96 & 61.49 \\
6 & 633 & 296.930 & 347.238 & 1.1694 & -2.25 & 0.66 & 11.73 & 61.49 \\
8 & 139 & 296.789 & 345.395 & 1.1637 & -1.78 & 1.12 & 6.07 & 61.99 \\
10 & 662 & 296.780 & 346.153 & 1.1663 & -1.53 & 2.07 & 6.96 & 61.85 \\
12 & 270 & 296.851 & 344.544 & 1.1606 & -0.94 & 2.41 & 5.32 & 62.00 \\
14 & 43 & 296.754 & 343.172 & 1.1564 & -0.47 & 3.00 & 4.71 & 61.33 \\
16 & 116 & 296.534 & 342.997 & 1.1567 & -0.52 & 4.83 & 4.64 & 62.07 \\
18 & 38 & 296.084 & 339.757 & 1.1475 & -0.63 & 5.34 & 4.30 & 62.50 \\
$\ge 20$  & 30 & 296.073 & 338.392 & 1.1429 & -0.13 & 10.39 & 3.43 & 63.13
\end{tabular}
\caption{Statistics for worst-case inverse problems.}\label{tab:wc}
\end{table}

Note that NomReg is decreasing with increasing WC values, i.e., when only little uncertainty is required to modify the instance such that the nominal solution is not the optimal regret solutions, then the regret of the nominal solution tends to be higher. Looking at the ratio between NomReg and Reg, we see that the quality of the nominal solution improves for larger values of WC. As the nominal solution is sometimes used as the baseline heuristic for solution algorithms (see, e.g., \cite{midpoint}), this may lead to structural insight on the performance of such algorithms for different instance classes.

The computation time for WC increases with the resulting objective value, while computation times for BC decreases for the respective instance classes. WC and BC values are connected, with instances having small WC values also tending to have large BC values. That is, if only little uncertainty is required to disturb the nominal solution, then the largest possible uncertainty set for which it is optimal also tends to be smaller. Our results do not show a significant increase in computation time for RegT, depending on WC.

Summarizing, we find that the WC value is able to categorize instances according to the relative performance of the nominal solution. While the ratio of NomReg and Reg is easier to compute than WC, if offers structural insight on why these instances behave differently, and can be used to structure benchmark sets as an example application.

\section{Conclusion and Further Research}
\label{sec:con}

In classic robust optimization problems, one aims at finding a solution that performs well for a given uncertainty set of possible parameter outcomes. In this paper, we considered a more general problem where we assume only the shape of the uncertainty set to be given, but not its actual size.

The resulting variable-sized min-max robust optimization problem is analyzed for different uncertainty sets, and results are applied to the shortest path problem. In a brief case study, we demonstrated the value of alternative solutions to the decision maker, which can be found in little computation time.

As a special case of variable-sized uncertainty, we considered inverse problems with min-max regret objective. To solve such problems, mixed-integer programming formulations were derived. Inverse robust optimization can also be applied to give structural insight to robust optimization instances, which was demonstrated with experimental data. 

Our research is the first of its kind, with possible applications in decision support, sensitivity analysis, and benchmarking. 

Future research will consider inverse problems for min-max regret, and more complex uncertainty sets than hyperboxes. In particular, $\Gamma$ uncertainty or ellipsoidal uncertainty might be considered. Furthermore, models and algorithms to extend our ideas to solutions which are not optimal for the nominal scenario will be helpful. On possible approach to this end is to measure the biggest difference of the regret of the given solution to the best possible regret for varying uncertainty sizes.

\bibliographystyle{alpha}
\bibliography{references}

\end{document}